\documentclass[12pt]{amsart}
\usepackage{amssymb}
\usepackage{enumerate}
\usepackage{graphicx}
\makeatletter
\@namedef{subjclassname@2010}{%
  \textup{2010} Mathematics Subject Classification}
\makeatother
\newtheorem{thm}{Theorem}[section]
\newtheorem{corollary}[thm]{Corollary}
\newtheorem{lemma}[thm]{Lemma}
\newtheorem{proposition}[thm]{Proposition}
\theoremstyle{definition}
\newtheorem{definition}[thm]{Definition}
\newtheorem{remark}[thm]{Remark}
\newtheorem{example}[thm]{Example}
\newcommand{\Z}{\operatorname{\mathbb{Z}}}
\newcommand{\Q}{\operatorname{\mathbb{Q}}}
\newcommand{\m}{\operatorname{\mathfrak{m}}}
\newcommand{\e}{\operatorname{Ext}}
\newcommand{\Ann}{\operatorname{Ann}}

\begin{document}
\baselineskip=17pt

\title[Matlis' semi-regularity in trivial ring extensions]{Matlis' semi-regularity in trivial ring extensions issued from integral domains}

\author[K. Adarbeh]{Khalid Adarbeh}
\address{Department of Mathematics and Statistics\\
King Fahd University of Petroleum and Minerals\\
Dhahran 31261, Saudi Arabia}
\email{khalidwa@kfupm.edu.sa}

\author[S. Kabbaj]{Salah Kabbaj $^{(\star)}$}
\address{Department of Mathematics and Statistics\\
King Fahd University of Petroleum and Minerals\\
Dhahran 31261, Saudi Arabia}
\email{kabbaj@kfupm.edu.sa}

\thanks{$^{(\star)}$ Corresponding author}

\date{}

\begin{abstract}
This paper contributes to the study of homological aspects of trivial ring extensions  (also called Nagata idealizations). Namely, we investigate the transfer of the notion of (Matlis') semi-regular ring (also known as IF-ring) along with related concepts, such as coherence, in trivial ring extensions issued from integral domains. All along the paper, we put the new results in use to enrich the literature with new families of examples subject to semi-regularity.
\end{abstract}



\maketitle

\section{Introduction}

\noindent Throughout, all rings considered are commutative with identity and all modules are unital. A ring $R$ is coherent if every finitely generated ideal of $R$ is finitely presented. The class of coherent rings includes strictly the classes of Noetherian rings, von Neumann regular rings (i.e., every module is flat), valuation rings, and semi-hereditary rings (i.e., every finitely generated ideal is projective). During the past three decades, the concept of coherence developed towards a full-fledged topic in commutative algebra under the influence of homology; and several notions grew out of coherence (e.g., finite conductor property, quasi-coherence, $v$-coherence, and $n$-coherence). For more details on coherence see please \cite{G1,G2} and for coherent-like properties see, for instance, \cite{KM1,KM2}.

In 1982, Matlis proved that a ring $R$ is coherent if and only if $\hom_{R}(M,N)$ is flat for any injective $R$-modules $M$ and $N$ \cite[Theorem 1]{M82}. In 1985, he defined a ring $R$ to be semi-coherent if $\hom_{R}(M,N)$ is a submodule of a flat $R$-module for any injective $R$-modules $M$ and $N$. Then, inspired by this definition and von Neumann regularity, he defined a ring to be semi-regular if any module can be embedded in a flat module (or, equivalently; if every injective module is flat) \cite{M85}. He then proved that semi-regularity is a local property in the class of coherent rings \cite[Proposition 2.3]{M85}.
Moreover, he proved that in the class of reduced rings, von Neumann regularity collapses to
semi-regularity \cite[Proposition 2.7]{M85}; and under Noetherian assumption, semi-regularity equals the self-injective property; i.e., $R$ is quasi-Frobenius if and only if $R$ is semi-regular and Noetherian \cite[Proposition 3.4]{M85}. Beyond Noetherian settings, examples of semi-regular rings arise as factor rings of Pr\"ufer domains over nonzero finitely generated ideals \cite[Proposition 5.3]{M85}. It is worth noting, at this point, that the notion of semi-regular ring was briefly mentioned by Sabbagh (1971) in \cite[Section 2]{Sabb} and studied in non-commutative settings by Jain (1973) in \cite{Jain}, Colby (1975) in \cite{Col}, and Facchini \& Faith (1995) in \cite{FF}, among others, where it was always termed as IF ring. Also, it was extensively studied -under IF terminology- in (commutative) valuation settings by Couchot in \cite{Cou2003,Cou2009,Cou2012}. Finally, recall that an $R$-module $E$ is fp-injective (also called absolutely pure) if $\e^{1}_{R}(M,E)=0$ for every finitely presented $R$-module $M$ \cite[IX-3]{FuSa}; and $R$ is self fp-injective if it is fp-injective over itself. Also, $R$ is semi-regular if and only if $R$ is self fp-injective and coherent \cite[Theorem 3.10]{Jain} or \cite[Theorem 2]{Col}.

For a ring $A$ and an $A$-module $E$, the trivial ring extension of $A$ by $E$ is the ring $R:= A\ltimes E$ where the underlying group is $A \times E$ and the multiplication is defined by $(a,e)(b,f) = (ab, af+be)$. The ring $R$ is also called the (Nagata) idealization of $E$ over $A$ and is denoted by  $A(+)E$. This construction was first introduced, in 1962, by Nagata \cite{Na} in order to facilitate interaction between rings and their modules and also to provide various families of examples of commutative rings containing zero-divisors. The literature abounds of papers on trivial extensions dealing with the transfer of ring-theoretic notions in various settings of these constructions (see, for instance, \cite{AJK,BKM,DS,Fo,Go,GoMaPh,Gull,Kou,Le,Olb3,Olb4,PR,Po,Rei,Roo,Sal}). For more details on commutative trivial extensions (or idealizations), we refer the reader to Glaz's and Huckaba's respective books \cite{G1,H}, and also D. D. Anderson \& Winders relatively recent and comprehensive survey paper \cite{AW}.

 This paper contributes to the study of homological aspects of trivial ring extensions. Namely, we investigate the transfer of the notion of (Matlis') semi-regular ring (also known as IF-ring) along with related concepts, such as coherence, in trivial ring extensions issued from integral domains. All along the paper, we put the new results in use to enrich the literature with new families of examples subject to semi-regularity.

 For the reader's convenience, Figure~\ref{D1} displays a diagram of implications summarizing the relations among the main notions involved in this work.

\begin{figure}[ht]
\centering
\[\setlength{\unitlength}{.6mm}
\begin{picture}(120,80)(0,-85)
\put(0,-10){\vector(2,-1){40}}
\put(0,-10){\vector(0,-1){20}}
\put(40,-30){\vector(0,-1){20}}
\put(40,-30){\vector(2,-1){40}}
\put(40,-30){\vector(-2,-1){40}}
\put(0,-30){\vector(2,-1){40}}
\put(0,-50){\vector(0,-1){20}}
\put(80,-50){\vector(0,-1){20}}
\put(120,-30){\vector(-2,-1){40}}
\put(120,-30){\vector(0,-1){20}}
\put(120,-50){\vector(-2,-1){40}}
\put(40,-50){\vector(2,-1){40}}
\put(40,-50){\vector(-2,-1){40}}

\put(0,-10){\circle*{1.2}}\put(0,-8){\makebox(0,0)[b]{\scriptsize Semisimple}}
\put(40,-30){\circle*{1.2}}\put(42,-29){\makebox(0,0)[b]{\scriptsize Quasi-Frobenius}}
\put(0,-30){\circle*{1.2}} \put(-2,-30){\makebox(0,0)[r]{\scriptsize von Neumann regular}}
\put(0,-50){\circle*{1.2}} \put(-4,-50){\makebox(0,0)[r]{\scriptsize Self-injective}}
\put(40,-50){\circle*{1.2}}\put(47,-48){\makebox(0,0)[b]{\scriptsize\bf Semi-regular}}
\put(80,-50){\circle*{1.2}} \put(85,-46){\makebox(0,0)[b]{\scriptsize  Noetherian}}
\put(0,-70){\circle*{1.2}} \put(-2,-70){\makebox(0,0)[r]{\scriptsize Self fp-injective}}
\put(80,-70){\circle*{1.2}}\put(78,-70){\makebox(0,0)[r]{\scriptsize  Coherent}}
\put(120,-30){\circle*{1.2}}\put(122,-30){\makebox(0,0)[l]{\scriptsize  Dedekind}}
\put(120,-50){\circle*{1.2}}\put(122,-50){\makebox(0,0)[l]{\scriptsize  Pr\"ufer}}
\end{picture}\]
\caption{\small A ring-theoretic perspective for semi-regularity}\label{D1}
\end{figure}
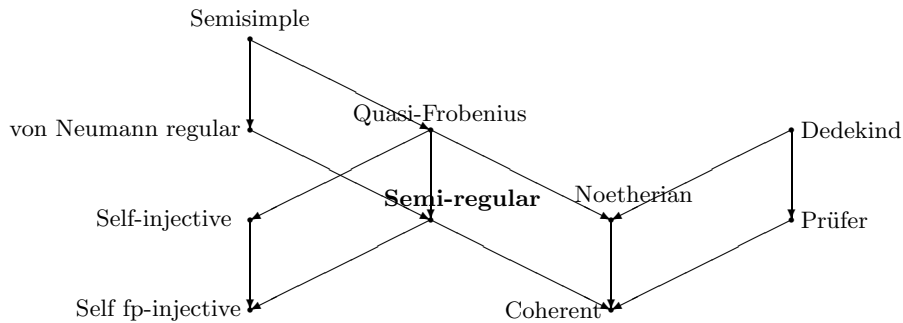

\section{Main result}\label{d}

\noindent We investigate the transfer of semi-regularity to trivial ring extensions issued from domains. We first state some preliminary results which will make up the proof of the main result of this paper (Theorem~\ref{d:thm1}).

Recall that a module over a domain is divisible if each element of the module is divisible by every nonzero element of the domain \cite{Ro}. The first lemma asserts that fp-injectivity and, a fortiori, divisibility of the module $E$ are necessary conditions for the trivial extension $A\ltimes~E$ to inherit semi-regularity.

\begin{lemma}\label{d:lem1}
Let $A$ be a ring, $E$ an $A$-module, and $R:=A\ltimes~E$. If $R$ is  self fp-injective, then $E$ is fp-injective. In particular, if $A$ is a domain and $R$ is semi-regular, then $E$ is divisible.
\end{lemma}

\begin{proof}
Let $M:=\sum_{1\leq i\leq n} Am_{i}$ be a finitely generated submodule of $A^{n}$, for some positive integer $n$, and let $f: M \longrightarrow E$ be an $A$-map. One can identify $R^{n}$ with $A^{n}\ltimes~E^{n}$ as $R$-modules under natural scalar multiplication. Consider the finitely generated submodule of $R^{n}$ given by $N:=\sum_{1\leq i\leq n} R(m_{i},0)$ along with the $R$-maps
$$N \stackrel{p}\twoheadrightarrow M \stackrel{f}\longrightarrow E \stackrel{u}\hookrightarrow R$$
where $p$ is defined by
$$p\big(\sum_{1\leq i\leq n}(a_{i},e_{i})(m_{i},0)\big)=\sum_{1\leq i\leq n} a_{i}m_{i}$$
and $u$ is the canonical embedding. Then, $g:=u\circ f\circ p$ extends to $R^{n}$ via $\overline{g}$, since $R$ is self fp-injective. It follows that $f$ extends to $A^{n}$ via the $A$-map
$$\overline{f} : A^{n} \stackrel{i}\hookrightarrow R^{n} \stackrel{\overline{g}}\longrightarrow R \stackrel{\pi} \twoheadrightarrow E$$
where $i$ is the canonical embedding and $\pi$ is the canonical surjection. Therefore, $E$ is fp-injective \cite[Theorem IX-3.1]{FuSa}. The second statement of the lemma is straightforward since a semi-regular ring is self fp-injective; and an fp-injective module is divisible.
\end{proof}

\begin{remark}\label{d:rem1}
The second statement of the lemma is still valid if $A$ is an arbitrary ring (i.e., possibly with zero-divisors) and divisibility of $E$ is taken over all non zero-divisors of $A$.
\end{remark}

The next lemma shows that divisibility of the module $E$ controls the finitely generated ideals of the trivial extension $R:=A\ltimes~E$.

\begin{lemma}\label{d:lem2}
Let $A$ be a domain, $E$ a divisible $A$-module, and $R:=A\ltimes~E$. Then, for any finitely generated ideal $\mathcal{I}$ of $R$,
either $\mathcal{I}=I\ltimes E$ for some nonzero finitely generated ideal $I$ of $A$ or $\mathcal{I}=0\ltimes E'$ for some finitely generated submodule $E'$  of $E$.
\end{lemma}

\begin{proof}
First, note that if $E'$ is a  finitely generated submodule  of $E$, then $0\ltimes E'$ is a finitely generated ideal of $R$. Also, let  $I:=\sum_{1\leq i\leq n} Aa_{i}$ with $0\neq a_{i}\in A$ for all $i$ and let $e\in E$. Then, by divisibility, $e=a_{1}e'$ for some $e'\in E$ and, hence,   $(0,e)=(a_{1},0)(0,e')$. It follows that
$$I\ltimes E=\sum_{1\leq i\leq n}(a_{i},0)R.$$
That is, $I\ltimes E$ is a finitely generated ideal of $R$.

Next, let $\mathcal{I}=\sum_{1\leq i\leq n}(x_{i},e_{i})R$ with $x_{i}\in A$ and $e_{i}\in E$ for $i=1,\dots,n$. If $x_{i}=0$ for all $i$, then
$$\mathcal{I}=\sum_{1\leq i\leq n} 0\ltimes Ae_{i}=0\ltimes E'$$
with $E':=\sum_{1\leq i\leq n} Ae_{i}$, as desired. Next, assume the $x_{i}$'s are not all null and, mutatis mutandis,  let $r\in\{1,\dots,n\}$ such that $x_{i}\neq 0$ for $i\leq r$ and $x_{i}=0$ for $i\geq r+1$. We claim that $\mathcal{I}=I\ltimes E$ with $I:=\sum_{1\leq i\leq r} Ax_{i}$. Indeed, $\forall\ i\in\{1,\dots,r\}$ and $\forall\ j\in\{r+1,\dots,n\}$, we have
$$(x_{i},e_{i})R\subseteq Ax_{i}\ltimes (Ex_{i}+Ae_{i})\subseteq I\ltimes E,$$
$$(x_{j},e_{j})R=0\ltimes Ae_{j}\subseteq I\ltimes E$$
so that $\mathcal{I}\subseteq I\ltimes E$. For the reverse inclusion, let $z:=(\sum_{1\leq i\leq r}a_{i}x_{i},e)\in I\ltimes E$. We can write
$$z:=(a_{1}x_{1},e)+\sum_{2\leq i\leq r}(a_{i}x_{i},0).$$
So, it suffices to show that $(a_{i}x_{i},e)\in(x_{i},e_{i})R$, for any given $e\in E$ and $i\in\{1,\dots, r\}$. This holds if there is $e'\in E$ such that
$$e=x_{i}e'+a_{i}e_{i}.$$
Indeed, recall at this point that $E$ is divisible and suppose $e=0$. If $a_{i}e_{i}=0$, take $e':=0$; and if $a_{i}e_{i}\neq 0$, then $a_{i}e_{i}=x_{i}e'_{i}$ for some $e'_{i}\in E$ and hence take $e':=-e'_{i}$. Suppose $e\neq 0$ and let $e=x_{i}e"_{i}$ for some $e"_{i}\in E$. If $a_{i}e_{i}=0$, take $e':=e"_{i}$; and if $a_{i}e_{i}\neq 0$, take $e':=e"_{i}-e'_{i}$, proving the claim.
\end{proof}

\begin{remark}\label{d:rem2}
Notice that the converse of the above lemma is always true; namely, if all finitely generated ideals of $R$ have the two aforementioned forms, then $E$ is divisible. For, let $x$ be a nonzero element of $A$. Then, $(x,0)R=xA\ltimes xE$ is a finitely generated ideal of $R$ with $xA\neq 0$, which forces $E=xE$.
\end{remark}

Next, we examine the transfer of coherence in trivial extensions of domains by divisible modules. In this vein, we will use Fuchs-Salce's definition of a coherent module; that is, all its finitely generated submodules are finitely presented \cite[Chapter IV]{FuSa} (i.e., the module itself doesn't have to be finitely generated). In Bourbaki, such a module is called ``pseudo-coherent" \cite{Bou} and Wisbauer called it  ``locally coherent" \cite{Wis}.

We first isolate the simple case when $A$ is trivial. Namely, if  $A:=k$ is a field and $E$ is a $k$-vector space, then a combination of \cite[Theorem 2.6]{KM2} and \cite[Theorem 4.8]{AW} yields: ``\emph{$k\ltimes~E$ is coherent if and only if $k\ltimes~E$ is Noetherian if and only if $\dim_{k}E<\infty$.}" The next result handles the case when $A$ is a non-trivial domain.

\begin{proposition}\label{d:lem3}
Let $A$ be a domain which is not a field, $E$ a divisible $A$-module, and $R:=A\ltimes~E$. Then, $R$ is coherent if and only if  $A$ is coherent, $E$ is torsion coherent, and $\Ann_{E}(x)$ is finitely generated for all $x\in A$.
\end{proposition}

\begin{proof}
Assume $R$ is coherent. Then so are its retract $A$ by \cite[Theorem 4.1.5]{G1} and $E$ by Glaz's remark following \cite[Theorem 4.4.4]{G1} in page 146. Now, assume there is a torsion-free element $e\in E$ and let $0\neq a\in A$. Then
$$\Ann_{R}(0,e)=\Ann_{A}(e)\ltimes E=0\ltimes E$$
is a finitely generated ideal of $R$. So $E$ is a finitely generated $A$-module. Let $e_{1},\dots,e_{n}$ be a minimal generating set for $E$. By divisibility assumption, we obtain $e_{1}=a\sum_{1\leq i\leq n} a_{i}e_{i}$,
for some $a_{1},\dots,a_{n}\in A$. If $1-aa_{1}\neq 0$, then
$$e_{1}=(1-aa_{1})\sum_{1\leq i\leq n} b_{i}e_{i}$$
for some $b_{1},\dots,b_{n}\in A$, forcing
$$e_{1}\in\sum_{2\leq i\leq n} Ae_{i}$$
which is absurd. So, necessarily, $1-aa_{1}=0$. It follows that $A$ is a field, the desired contradiction. Hence, $E$ is a torsion module. Finally, let $0\neq x\in A$. Then, $\Ann_{R}(x,0)=0\ltimes\Ann_{E}(x)$ is finitely generated in $R$. So $\Ann_{E}(x)$ is a finitely generated submodule of $E$.

Conversely, we first show that the intersection of any two finitely generated ideals of $R$ is finitely generated. Let $I_{1}$ and $I_{2}$ be two nonzero finitely generated ideals of $A$ and let  $E_{1}$ and $E_{2}$ be two finitely generated submodules of $E$. Since $A$ is a coherent domain, $I_{1}\cap I_{2}$ is a nonzero finitely generated ideal of $A$. By Lemma~\ref{d:lem2}, $$(I_{1}\ltimes E)\cap (I_{2}\ltimes E)=(I_{1}\cap I_{2})\ltimes E$$
is a finitely generated ideal of $R$. Further, obviously,
$$(I_{1}\ltimes E)\cap (0\ltimes E_{1})=0\ltimes E_{1}$$
is finitely generated. Moreover, since $E$ is coherent, $E_{1}\cap E_{2}$ is a finitely generated submodule of $E$ \cite[(D)--Page 128]{FuSa}. Hence,
$$(0\ltimes E_{1})\cap (0\ltimes E_{2})=0\ltimes (E_{1}\cap E_{2})$$
is a finitely generated ideal of $R$. In view of Lemma~\ref{d:lem2}, we are done. By \cite[Theorem 2.3.2(7)]{G1}, it remains to show that $\Ann_{R}(x,e)$ is finitely generated for any $(x,e)\in R$. Indeed, if $x\neq 0$, then
$$\Ann_{R}(x,e)=0\ltimes\Ann_{E}(x)$$
is finitely generated in $R$ (since by hypothesis $\Ann_{E}(x)$ is finitely generated). Next, assume $x=0$. In view of the exact sequence
$$0\rightarrow\Ann_A(e)\rightarrow A\rightarrow Ae\rightarrow  0,$$ since $E$ is torsion coherent, $\Ann_A(e)$ is a nonzero finitely generated ideal of $A$. By Lemma~\ref{d:lem2}, $$\Ann_{R}(0,e)=\Ann_A(e)\ltimes E$$
is a finitely generated ideal of $R$, completing the proof of the proposition.
\end{proof}

In the above result, the assumption ``$\Ann_{E}(x)$ is finitely generated for all $x\in A$" is not superfluous in presence of the other assumptions, as shown by the next example.
Throughout, for a domain $A$, $Q(A)$ will denote its quotient field.

\begin{example}\label{d:ex1}
Let $A$ be a coherent domain which is not a field (e.g., any non-trivial Pr\"ufer domain) and $E:=\bigoplus_{n\geq0} E_{n}$ with $E_{n}:=Q(A)/A$. Then, $E$ is a divisible coherent $A$-module \cite[(C)--Page 37 \& (B)--Page 128]{FuSa} and, clearly, $E$ is torsion. However, the condition ``$\Ann_{E}(x)$ is finitely generated for all $x\in A$" does not hold. For, let $x$ be any nonzero nonunit element of $A$. Then, one can easily check that
$$\Ann_{E}(x)=\bigoplus_{n\geq0} \overline{(1/x)}$$
which is not finitely generated.
\end{example}

In order to proceed further, we need to extend, to an $A$-module, Matlis' double annihilator condition in a ring $A$; i.e., $\Ann_{A}(\Ann_{A}(I))=I$, for each finitely generated ideal $I$ of $A$ \cite[Section 4, Definition]{M85}.

\begin{definition}\label{d:def1}
Let $A$ be a ring. An $A$-module $E$ is said to satisfy the double annihilator condition (in short, DAC) if the following two assertions hold:
\begin{enumerate}
\item[(DAC1)] $\Ann_{A}(\Ann_{E}(I))=I$, for every finitely generated ideal $I$ of $A$.
\item[(DAC2)]  $\Ann_{E}(\Ann_{A}(E'))=E'$, for every finitely generated submodule $E'$ of $E$.
\end{enumerate}
\end{definition}

Obviously, this definition coincides with Matlis' double annihilator condition when $E=A$. Moreover, all these conditions are unrelated in general, as shown by the following basic examples.

\begin{example}\label{d:ex2}
Let $A$ be a ring and $E$ a nonzero $A$-module.
\begin{enumerate}
\item Assume $A:=K$ is a field. Then, $E$ satisfies (DAC1). Moreover, $E$ satisfies (DAC2) if and only if $\dim_{K}(E)=1$. For,  the first statement is straightforward and the second holds as $\Ann_{E}(\Ann_{K}(e))=E$, for any nonzero $e\in E$.

\item Assume $(A,\m)$ is local and $E:=A/\m$. Then, $E$ satisfies (DAC2). Moreover, $E$ satisfies (DAC1) if and only if $l(\m)=1$. Indeed, the first statement is straight since $E$ has no nonzero proper submodules. The second statement holds since $\Ann_{A}(\Ann_{E}(x))=\m$, for any $x\in\m$.

\item Assume $A$ satisfies Matlis' double annihilator condition (e.g., semi-regular) and $E$ has a torsion-free element. Then, $E$ satisfies (DAC) if and only if $E\cong A$. This is true since  $\Ann_{E}(\Ann_{A}(e))=E$, for any given torsion-free element $e\in E$.
\end{enumerate}
\end{example}

We also need the next lemma which characterizes the double annihilator condition in a trivial ring extension via the (DAC) property of its divisible module.

\begin{lemma}\label{d:lem5}
Let $A$ be a domain, $E$ a divisible $A$-module, and $R:=A\ltimes~E$. Then, $R$ satisfies Matlis' double annihilator condition if and only if $E$ satisfies (DAC).
\end{lemma}

\begin{proof}
First, notice that $\Ann_{A}(\Ann_{E}(0))=\Ann_{A}(E)=0$, since $aE=E,\ \forall\ 0\neq a\in A$. Now, by Lemma~\ref{d:lem2}, the finitely generated ideals of $R$ have the forms $I\ltimes E$ and $0\ltimes E'$, where $I$ is a nonzero finitely generated ideal of $A$ and $E'$ is a finitely generated submodule of $E$. Moreover, one can easily check that
$$\Ann_{R}(I\ltimes E)=0\ltimes \Ann_{E}(I)$$
and
$$\Ann_{R}(0\ltimes E')=\Ann_{A}(E')\ltimes E.$$
It follows that
$$\Ann_{R}(\Ann_{R}(I\ltimes E))=\big(\Ann_{A}(\Ann_{E}(I))\big)\ltimes E$$
and
$$\Ann_{R}(\Ann_{R}(0\ltimes E'))=0\ltimes\big(\Ann_{E}(\Ann_{A}(E'))\big),$$
leading to the conclusion.
\end{proof}

Finally, we are ready to state the main theorem of this section on the transfer of semi-regularity to trivial ring extensions.

\begin{thm}\label{d:thm1}
Let $A$ be a domain and $E$ an $A$-module. Then, $R:=A\ltimes~E$ is semi-regular if and only if either $A$ is a field with $E\cong A$ or $A$ is a coherent domain, $E$ is a divisible (resp., fp-injective) torsion coherent module which satisfies (DAC), and $\Ann_{E}(x)$ is finitely generated for all $x\in A$.
\end{thm}

\begin{proof}
Let us first isolate the simple case when $A$ is trivial. Namely, let  $A:=k$ be a field and $E$ a nonzero $k$-vector space. Then we have, via Example~\ref{d:ex2}(1), that $\dim_{k}E=1$ if and only if $k\ltimes~E$ satisfies (DAC) if and only if $k\ltimes~E$ is semi-regular. Now, assume that $A$ is a domain which is not a field. Combine Lemma~\ref{d:lem1}, Proposition~\ref{d:lem3}, and Lemma~\ref{d:lem5} with Matlis' result  that ``\emph{a ring is semi-regular if and only if it is coherent and satisfies the double annihilator condition (on finitely generated ideals)}" \cite[Proposition 4.1]{M85}.
\end{proof}

 At this point, recall that a nonzero fractional ideal $I$ of a domain $A$ is divisorial if $I=I_{v}:=(I^{-1})^{-1}$. A domain is called divisorial if all its nonzero (fractional) ideals are divisorial. Divisorial domains have been studied by, among others, Bass \cite{Bass} and Matlis \cite{M68} for the Noetherian case, Heinzer \cite{Hein} for the integrally closed case, Bastida-Gilmer \cite{BaGi} for the transfer to $D+M$ constructions, and Bazzoni \cite{Baz} for more general settings. It is worthwhile recalling that a domain in which all finitely generated ideals are divisorial is not necessarily divisorial
 \cite[Example 2.11]{Baz}. Finally, recall that a domain $A$ is totally divisorial if every overring of $A$ is a divisorial domain; and $A$ is stable if every nonzero ideal of $A$ is projective over its ring of endomorphisms \cite{FuSa,Olb2}. A domain $A$ is totally divisorial if and only if $A$ is a stable divisorial domain \cite[Theorem 3.12]{Olb2}.

As an application of Theorem~\ref{d:thm1}, the next corollary will allow us to enrich the literature with new families of examples subject to semi-regularity. In the sequel, if $I$ and $J$ are (fractional) ideals of a domain $A$, let
$$(I:J)=\big\{x\in Q(A) \mid xJ\subseteq I\big\},$$
$$(I:_{A}J)=\big\{a\in A \mid aJ\subseteq I\big\}.$$

\begin{corollary}\label{d:cor1}
Let $A$ be a coherent domain which is not a field and $I$ a nonzero finitely generated fractional ideal of $A$. Then:
\begin{enumerate}[\upshape (1)]
\item $A\ltimes~\frac{Q(A)}{I}$ is semi-regular if and only if $(I:(I:J))=J$ for each nonzero finitely generated (fractional) ideal $J$ of $A$.
\item In particular, $A\ltimes~\frac{Q(A)}{A}$ is semi-regular if and only if each nonzero finitely generated (fractional) ideal of $A$ is divisorial.
\end {enumerate}
\end{corollary}

\begin{proof}
(1) First, notice that $Q(A)$ is a coherent $A$-module since it is torsion-free \cite[IV-2, Lemma 2.5]{FuSa}.  Further, given any exact sequence of modules over a coherent ring $0\rightarrow M'\rightarrow M\rightarrow M"\rightarrow  0$, if any two of the modules $M'$, $M$, $M"$ are finitely presented, then so is the third \cite[IV-2, Exercise 2.5]{FuSa}. It follows that $E:=\frac{Q(A)}{I}$ is coherent, with $I$ regarded as a finitely generated submodule of $Q(A)$. Moreover, $E$ is clearly a divisible torsion module and $\Ann_{E}(x)=\overline{\frac{1}{x}I}$, for any nonzero $x\in A$. Therefore, by Theorem~\ref{d:thm1},
$A\ltimes~E$ is semi-regular if and only if $E$ satisfies (DAC). So, we just need to prove the following claim.
\medskip

{\bf Claim:} $\frac{Q(A)}{I}$ satisfies (DAC) if and only if $(I:(I:J))=J$ for each nonzero finitely generated (fractional) ideal $J$ of $A$.
\medskip

Indeed,  assume $(I:(I:J))=J$ for each nonzero finitely generated (fractional) ideal $J$ of $A$. Note first that for $J:=A$, we get $$A=(I:(I:A))=(I:I).$$ Next, let $\overline{J}$ be a nonzero finitely generated submodule of $E$; that is, $J$ is a nonzero finitely generated fractional ideal of $A$ containing $I$. Then $(I:J)\subseteq (I:I)=A$ and hence
$$\Ann_{A}(\overline{J})=A\cap(I:J)=(I:_{A}J)=(I: J).$$
Moreover, let $K$  be a nonzero finitely generated ideal of $A$. Then
$$\Ann_{E}(K)=\overline{(I:K)}.$$
Therefore, since $KI\subseteq I$, we obtain
$$\Ann_{A}(\Ann_{E}(K))=\Ann_{A}\left(\overline{(I:K)}\right)= (I:(I: K))=K$$
and
$$\Ann_{E}(\Ann_{A}(\overline{J}))=\overline{(I:(I:_{A} J))}=\overline{(I:(I:J))}=\overline{J}$$
proving the ``if" assertion.

Conversely, assume that $E$ satisfies (DAC) and let $0\neq a\in A$ such that $aI\subseteq A$. Since $\frac{Q(A)}{aI}\cong \frac{Q(A)}{I}$ as $A$-modules and $(aI:(aI:J))=(I:(I:J))$ for each $J$, we may assume without loss of generality that $I$ is an (integral) ideal of $A$. Then (DAC2), applied to $J:=A$, yields $$\overline{A}=\Ann_{E}(\Ann_{A}(\overline{A}))=\overline{(I:(I:_{A} A))}=\overline{(I:I)}$$
so that $A=(I:I)$. Now, let $J$ a be a nonzero finitely generated ideal of $A$. Then, via the basic fact $I\subseteq (I:J)$, (DAC1) yields
$$J=\Ann_{A}(\Ann_{E}(J))=\Ann_{A}\left(\overline{(I:J)}\right)= (I:_{A}(I: J))= (I:(I: J))$$
completing the proof of (1).

(2) Straightforward via (1) with $I:=A$ and the fact $(A:(A:J))=J_{v}$.

\end{proof}

The above proof revealed that ``$A\ltimes~\frac{Q(A)}{I}$ is semi-regular if and only if $\frac{Q(A)}{I}$ satisfies (DAC)." So, let $A$ be a coherent domain which is not a field and $I$ a nonzero finitely generated fractional ideal of $A$. By Lemma~\ref{d:lem1}, if $\frac{Q(A)}{I}$ satisfies (DAC), then it is fp-injective. We don't know if the converse holds in general.

A von Neumann regular ring is a reduced semi-regular ring \cite[Proposition 2.7]{M85}. Matlis noticed that ``(von Neumann) regular rings and quasi-Frobenius rings are seen to have a common denominator of definition--they are both extreme examples of semi-regular rings." Next, we provide various examples of semi-regular trivial ring extensions which are neither von Neumann regular (since not reduced) nor quasi-Frobenius (since not Noetherian).

\begin{example}\label{d:ex3}
Let $A$ be a coherent domain which is not a field and let $R:=A\ltimes~\frac{Q(A)}{A}$. Note that $R$ is not Noetherian since $\frac{Q(A)}{A}$ is not finitely generated.
\begin{enumerate}
\item Assume $A$ is integrally closed. Then:
\begin{center}
$R$ is semi-regular $\Longleftrightarrow$ $A$ is Pr\"ufer.
\end{center}
Indeed, combine Corollary~\ref{d:cor1} with the fact that every invertible ideal is divisorial and Krull's result that ``\emph{an integrally closed domain in which all nonzero finitely generated ideals are divisorial is Pr\"ufer}" (cf. \cite[Proof of Theorem 5.1]{Hein}). For an original example, take $A$ to be any non-trivial Pr\"ufer domain (e.g., $A:=\Z+X\Q[X]$).

\item If $A$ is a divisorial domain, then $R$ is semi-regular by Corollary~\ref{d:cor1}. For an original example, take $A$ to be any pseudo-valuation domain issued from a valuation domain $(V,M)$ with $M$ finitely generated and $[\frac{V}{M}:k]=2$. Then, $A$ is a (non-integrally closed)  divisorial domain
    \cite[Theorem 2.1 \& Corollary 4.4]{BaGi}, which is coherent \cite[Theorem 3]{DP} or
    \cite[Theorem 3]{BR}.

\item Next, we provide a non-integrally closed non-divisorial domain $A$ in which every finitely generated ideal is divisorial; and hence $R$ is semi-regular by Corollary~\ref{d:cor1}. Indeed, let $D$ be a non-integrally closed pseudo-valuation domain which is divisorial and coherent (e.g., take $D$ to be the domain $A$ of (2) above) and let $K$ be its quotient field. By \cite[Theorem 2.6]{Olb1}, $D$ is not stable and hence not totally divisorial by \cite[Theorem 3.12]{Olb2}. Let $V$ be a valuation domain of the form $K+M$ and let $A:=D+M$. Then, $A$ is a non-integrally closed non-divisorial domain \cite[Theorem 2.1 \& Corollary 4.4]{BaGi} which is coherent \cite[Theorem 3]{DP} or
    \cite[Theorem 3]{BR}. Moreover, since $D$ is divisorial, every finitely generated ideal of $A$ is divisorial by
    \cite[Theorem 2.1(k) \& Theorem 4.3]{BaGi}.
\end{enumerate}
\end{example}

Other original examples stem from Pr\"ufer domains via Corollary~\ref{d:cor1}. For instance, for any Pr\"ufer domain $A$ and non-zero finitely generated (fractional) ideal $I$ of $A$, the trivial ring extension $A\ltimes~\frac{Q(A)}{I}$ is semi-regular. Indeed, let $J$ be a non-zero finitely generated ideal of $A$. Then, the basic facts $(IJ^{-1})J\subseteq I$ and $J(I:J)\subseteq I$, yield $(I:J)=IJ^{-1}$. It follows that $(I:(I:J))=(I:IJ^{-1})=I(IJ^{-1})^{-1}=J_{v}=J$, as desired.

Observe that for an example of a module $E$ which is not of the form $Q(A)/I$, one may appeal to non-standard uniserial modules. From \cite[X-3]{FuSa}, a uniserial module over a valuation domain with quotient  field $Q$ is \emph{standard} if it is isomorphic to $J/I$ for some ideals $0 \subseteq I \subseteq J \subseteq Q$. A uniserial module is \emph{non-standard} if it is not isomorphic to such a quotient. In this vein, recall that torsion-free uniserial modules are necessarily standard. Next, by \cite[Example VII-4.1 \& Theorem X-4.5 \& following comment]{FuSa}, let $A$ be a valuation domain for which there exists a divisible non-standard uniserial module $E$ whose non-zero elements have principal annihilators. Then, the trivial ring extension $R:=A\ltimes E$ is a chained ring that is not a homomorphic image of a valuation domain \cite[Theorem X-6.4]{FuSa}. Moreover, by \cite[Theorem 10]{Cou2003}, $R$ is semi-regular; indeed, let $0\neq e$ be a nonzero torsion element of $E$ with $\Ann_{A}(e)=aA$, for some $0\neq a\in A$. Since $E$ is divisible, it is easily seen that $\Ann_{R}(0,e)=\Ann_{A}(e)\ltimes E = aA\ltimes E = (a,0)R$, as desired.

\subsection*{Acknowledgements}
This work was funded by NSTIP Research Award \# 14-MAT71-04. We thank Fran\c{c}ois Couchot for his comments which helped to improve the quality of this paper. We also thank the referee for a very careful reading of the manuscript and useful suggestions.


\end{document}